\documentclass[12pt,oneside]{amsart}
\usepackage{xspace}
\usepackage{graphicx}
\usepackage{latexsym}
\usepackage{mathtools}
\usepackage{amsfonts,amsmath,amscd,amssymb}
\input{xy}
\xyoption{all}
\xyoption{curve}

\newcommand{\bC}{{\mathbb C}}
\newcommand{\bK}{{\mathbb K}}
\newcommand{\bI}{{\mathbb I}}
\newcommand{\bP}{{\mathbb P}}

\newcommand{\bZ}{{\mathbb Z}}
\newcommand{\bA}{{\mathbb A}}
\newcommand{\bD}{{\mathbb D}}

\newcommand{\bG}{{\mathbb G}}

\newcommand{\sA}{{\mathcal A}}
\newcommand{\sD}{{\mathcal D}}
\newcommand{\sM}{{\mathcal M}}
\newcommand{\sO}{{\mathcal O}}

\newcommand{\sT}{{\mathcal T}}
\newcommand{\sV}{{\mathcal V}}
\newcommand{\sX}{{\mathcal X}}

\newcommand{\g}{{\mathfrak g}}

\newcommand{\Gr}{{{\mbox{\rm --\,Grmod}}}} 
\newcommand{\Mo}{{{\mbox{\rm --\,Mod}}}} 
\newcommand{\End}{{{\mbox{\rm End}}}} 
\newcommand{\Ex}{{{\mbox{\rm Exp}_G}}} 
\newcommand{\Hom}{{{\mbox{\rm Hom}}}} 
 
\newcommand{\Qc}{{{\mbox{\rm --\,Qcoh}}}} 
\newcommand{\Qcoh}{{{\mbox{\rm Qcoh}}}} 
\newcommand{\To}{{{\mbox{\rm --\,Tors}}}} 
\newcommand{\Proj}{ {{\mbox{\rm Proj}}} } 
\newcommand{\Ke}{{{\mbox{\rm Ker}}}} 
\newcommand{\Ima}{{{\mbox{\rm Im}}}} 

\newcommand{\coke}{{{\mbox{\rm coker}}}}

\newcommand{\Ext}{{{\mbox{\rm Ext}}}}

\newcommand{\ve}{\ensuremath{\mathbf{e}}\xspace}
\newcommand{\vE}{\ensuremath{\mathbf{E}}\xspace}
\newcommand{\vf}{\ensuremath{\mathbf{f}}\xspace}
\newcommand{\vg}{\ensuremath{\mathbf{g}}\xspace}

\newcommand{\vv}{\ensuremath{\mathbf{v}}\xspace}
\newcommand{\vw}{\ensuremath{\mathbf{w}}\xspace}
\newcommand{\vx}{\ensuremath{\mathbf{x}}\xspace}

\newenvironment{subproof}{\vspace{-1ex}\par\noindent\normalsize {\em
    Proof of Claim}:}{{\hfill $\Box$}}

\newtheorem{theorem}{Theorem}
\newtheorem{prop}[theorem]{Proposition}
\newtheorem{lemma}[theorem]{Lemma}
\newtheorem{cor}[theorem]{Corollary}
\newtheorem{claim}{Claim}

\begin{document}
\title{D-modules and projective stacks}
\author{Karim El Haloui}
\email{kelhaloui@gmail.com}
\author{Dmitriy Rumynin}
\email{D.Rumynin@warwick.ac.uk}
\address{Dept. of Maths, University of Warwick, Coventry,
  CV4 7AL, UK}
\date{December 12, 2016}

\begin{abstract}
We study twisted D-modules on weighted projective stacks.
We determine for which values of the twist and the weight
the global sections functor is an equivalence, thus, proving
a version of Beilinson-Bernstein Localisation Theorem. 
\end{abstract}

\maketitle

A key observation in the proof of Kazhdan-Lusztig Conjecture by
Beilinson and Bernstein is that the (generalised) flag varieties 
$G/P$ are D-affine.
This is known as Beilinson-Bernstein Localisation Theorem.
So far these are the only known connected smooth projective D-affine
varieties.
In particular, Thomsen proves that a toric smooth projective D-affine
variety must be a product of projective spaces \cite{Tho}. 
On the other hand, Van den Bergh proves that weighted projective
spaces
are D-affine (they are singular)
\cite{MVB}. 

The goal of this paper is to re-examine the D-affinity of weighted projective spaces.
Instead of looking at them as singular varieties, we consider them as stacks.
We give a necessary and sufficient criterion for a weighted projective stack to be D-affine.
Our method of proof is also different: Van den Bergh uses Hodges-Smith Criterion for D-affinity \cite{HS},
while we do a direct calculation.

In section 1 we make general observations about D-affinity on varieties.
In section 2 we establish a technical framework for working with twisted D-modules
on a smooth projective stack.
In section 3 we use this framework to
study D-modules on weighted projective stacks.
\section*{Acknowledgement}
The authors would like to thank 
Mohamed Barakat,
Mark Bell, 
Gwyn Bellamy,
Michael Groechenig,  
Paul Smith and 
Jesper Funch Thomsen
for useful discussions.
The second author was partially supported by the Russian Academic Excellence Project `5--100' and by Leverhulme Foundation (grant RPG-2014-106). 

\section{D-modules on varieties}
We work with  a connected algebraic variety $X$ 
over an algebraically closed field
$\bK$
of characteristics zero in this section.
Let
$\sO_X$ 
be 
its sheaf of functions,
$\sD_X$ its sheaf of differential
operators, 
$D(X) = \sD_X (X)$ its global sections.
We consider the category of quasicoherent $\sD_X$-modules
$\sD_X\Qc$ and the category of modules over 
the globally defined
differential operators
$D(X)\Mo$. They are connected by the global sections functor
$$\Gamma: \sD_X\Qc \rightarrow D(X)\Mo.$$
$X$ is called {\em D-affine} if $\Gamma$ is an equivalence.
Affine varieties are D-affine but
the converse statement is not true:
the generalised flag variety $G/P$ is a smooth projective D-affine 
variety \cite{BB}. In the light of this result, it is interesting to pose the following question.

\noindent 
{\bf Question:} {\em Classify connected smooth projective D-affine
varieties.}

It would be interesting to find other examples of such varieties besides $G/P$.
Notice that any such example $X$ must have 
zero Hodge numbers $h^{0,m} (X)$ for $m>0$ because $\sO_X$ 
is a $\sD_X$-module, hence, has no higher cohomology. 
A glimmering hope for settling this question is the result of Thomsen who classified 
smooth toric D-affine
varieties \cite{Tho}.
Hereby we will explain that
some other classes of varieties 
will not give new examples.

Recall that a variety $X$ is homogeneous if a connected algebraic (not
necessarily
linear) group $G$ acts transitively on $X$. For a complete variety $X$
it is equivalent to asking that the automorphism group of $X$ acts
transitively
on $X$ \cite{SdS}. Such $X$ is necessarily smooth.

\begin{theorem}
Suppose $X$ is a homogeneous complete D-affine variety.
Then $X$ is isomorphic to a generalised flag variety.
\end{theorem}
\begin{proof}
By Borel-Remmert Theorem \cite{SdS}
$X$ is a product of a partial flag variety and an abelian variety $A$.
It remains to notice that $A$ is not 
D-affine because $R^{dim A} \Gamma (A, \sO_A)\neq 0$ 
by Serre's duality, unless $A$ is a point. 
This would imply that $R^{dim A} \Gamma (X, \sO_X)\neq 0$
that is impossible because $\sO_X$ is a $\sD_X$-module.
Thus, $A$ is a point and $X$ is a generalised flag variety.
\end{proof}

If $\bK=\bC$ is the field of complex numbers, this result can be slightly
improved.
\begin{theorem}
\label{C-homo}
Suppose $X$ is a complex complete D-affine variety
and the tangent sheaf $\sT_X$ is generated by global sections.
Then $X$ is isomorphic to a generalised flag variety.
\end{theorem}
\begin{proof}
Since $X$ is a complete algebraic variety,
the global (algebraic) 
vector fields $\g = \Gamma (\sT_X)$ form 
a finite dimensional Lie algebra \cite[p. 95]{Sha}. 
Let $G$ be an analytic connected simply-connected Lie group 
with Lie algebra~$\g$.
The group $G$ 
locally acts on $X$ by the second Lie Theorem  \cite[p. 23]{Akh}. 
Since $X$ is compact, each element $a\in\g$ defines
a one-parameter group $\gamma_a (t)$ of (global) diffeomorphisms of 
$X$ \cite[p. 20]{Akh}. Choosing a real basis $a_1,\ldots a_k$ of $\g$,
we can extend the assignment
$$
\Ex (t_1a_1)\cdot\Ex (t_2a_2)\cdot \ldots \Ex (t_ka_k)
\mapsto
\gamma_{a_1}(t_1)\gamma_{a_2}(t_2)\ldots\gamma_{a_k}(t_k) 
$$
to a global (real) analytic action of $G$ on $X$ \cite[p. 29]{Akh}.

Since $\sT_X$ is generated by global sections, 
each point $x\in X$ lies in the interior of its orbit $G\cdot x$.
Hence each point belongs to an open set, 
entirely within this point's orbit. By connectedness there is only one
orbit, hence, $X\cong G/H$ as analytic manifolds.

By Borel-Remmert Theorem \cite[p. 101]{Akh},
there exists an abelian variety $A$ such that
$X$ is an $A$-fibration over a generalised flag variety $Y$.
If $A$ is a point, we are done.
If $A$ is not a point, 
$R^{dim A} \Gamma (A, \sO_A)\neq 0$ 
by Serre's duality. Thus, the derived push-forward
$R(X\rightarrow Y)_\ast (\sO_X)$ has higher cohomology
and so does $\sO_X$. This is a contradiction.
\end{proof}

Observe that $\sT_X$ is not usually a $\sD_X$-module. 
This would require a flat connection on $\sT_X$ which is
quite rare. For instance,
abelian varieties admit a flat connection on $\sT_X$ 
as well as any other variety with a trivial tangent sheaf.
On the other hand, the only generalised flag variety
with a flat connection on $\sT_X$ is a point.

\begin{cor}
\label{C-conn}
If $X$ is complex complete D-affine variety and
$\sT_X$ is a $\sD_X$-module, then $X$ is the point.
\end{cor}

It would be interesting 
to extend Theorem~\ref{C-homo} and 
Corollary~\ref{C-conn} 
to varieties over an arbitrary
algebraically closed field $\bK$. 
Our proof does not work because we use analytic methods.

\section{D-modules on smooth projective stacks}

The theory of D-modules on stacks is known \cite{BD,DrGa}. 
Let $Y$ be a smooth algebraic variety with an action of 
an algebraic group $G$.
The quotient stack $[X]=[Y/G]$
admits the standard smooth atlas
$\xymatrix{
 G\times Y \ar@<-.2ex>[r]_a \ar@<.8ex>[r]^p & Y
}$
with the action and projection maps.
This atlas extends to a simplicial variety $\sX$
where $\sX_n = G^n \times Y$,
connected by the maps
$$
\sX (\varphi) : \sX_n \rightarrow \sX_m, \ \ 
\sX (\varphi) (g_1, \ldots g_n, y) =
(h_1, \ldots h_m, h_{m+1} \cdot y)
$$
where (with empty products equal to $1_G$)
$$
h_i = \prod_{j = \varphi(i-1)+1}^{\varphi(i)} g_j, \
h_{m+1} = \prod_{j = \varphi(m)+1}^{n} g_j
$$
for any non-decreasing function
$\varphi : [m]\rightarrow [n] =\{0,1, \ldots , n\}$.
For instance, these are the maps for the low dimensional faces
(recall that $\partial_i^n :[n-1] \rightarrow [n]$
is the increasing map without $i$ in the image):
$$
\sX (\partial_2^2) (g_1, g_2, y) = (g_1, g_2 \cdot y), \
\sX (\partial_1^2) (g_1, g_2, y) = (g_1 g_2 , y), 
$$
$$
\sX (\partial_0^2) (g_1, g_2, y) = (g_2 , y), \
\sX (\partial_1^1) (g, y) = g \cdot y, \
\sX (\partial_0^1) (g, y) = y.
$$

The category of quasicoherent D-modules on $[X]$
is equivalent to the category of
cosimplicial D-modules on $\sX$ \cite[6.2.2]{DrGa}.
Recall that a cosimplicial D-module $\sV$ consists
of a quasicoherent D-module $\sV_n$ on each $\sX_n$
together with an isomorphism of D-modules
$\sV(\varphi) : \sX(\varphi)^* \sV_m \rightarrow  \sV_{n}$
for any non-decreasing function
$\varphi : [m]\rightarrow [n]$
such that the simplicial identities hold. 

A cosimplicial D-module $\sV$ can be recovered (up to an isomorphism)
from the D-module $\sV_0$ and the D-module isomorphism
$$
\gamma: p^\ast \sV_0 = \sX(\partial_0^1)^\ast \sV_0
\xrightarrow{\sV(\partial_0^1)}
\sV_1
\xrightarrow{\sV(\partial_0^1)^{-1}}
\sX(\partial_1^1)^\ast \sV_0  = a^\ast \sV_0.
$$
The simplicial identities in dimension two
force {\em the cocycle condition} on the isomorphism $\gamma$,
coercing $(\sV_0,\gamma)$
into {\em a strongly equivariant D-module} on $Y$.
Vice versa, a strongly equivariant D-module on $Y$ can be extended
to a cosimplicial D-module on $\sX$.
This shows that the category of quasicoherent D-modules
on $[X]$ is equivalent to the category of strongly equivariant
quasicoherent D-modules on $Y$.

Further significant clarification is possible.
Consider a $\sD_Y$-module $M$ with a compatible $G$-action, i.e.,
$\,^g(dm) = \,^gd \, \,^gm$ for all $g\in G$, $d\in D$, $m\in M$.
This is sometimes called {\em a weakly equivariant D-module}.
Such a $G$-action yields an isomorphism of $\sO_G\otimes \sD_Y$-modules
$\gamma: p^\ast M \rightarrow a^\ast M$
\cite{Gin}.

The Lie algebra $\g$ of $G$ acts on $M$ in two ways:
via the differential of the action $\g\rightarrow \sD_Y$
and 
via the differential of the $G$-action.
These two actions coincide if and only if
$\gamma: p^\ast M \rightarrow a^\ast M$
is an isomorphism of $\sD_G\otimes \sD_Y$-modules
(note that $\sD_G\otimes \sD_Y\cong \sD_{G\times Y}$)
\cite{Gin}. This gives an alternative definition of
a strongly equivariant D-module.

The preceding discussion enables us (modulo equivalences of categories)
to define a quasicoherent $\sD_{[X]}$-module as a quasicoherent
strongly $G$-equivariant $\sD_Y$-module.

There are different notions of a projective stack, for instance,
a stack whose coarse moduli space is a projective variety.
Here we use a more restrictive notion:
a projective stack is a smooth closed substack of a weighted
projective stack \cite{Zho}. Let us spell it out.
Let $V=\bigoplus V_k$ be a positively graded 
$n+1$-dimensional $\bK$-vector space. Naturally
we treat it as a $\bG_m$-module with positive weights by
$\lambda \bullet \vv_k = \lambda^k \vv_k$ where $\vv_k \in V_k$.
Let $Y$ be a smooth closed $\bG_m$-invariant subvariety of $V\setminus
\{0\}$. We define {\em a projective stack} as the stack $[X]=[Y/\bG_m]$.
The G.I.T.-quotient $X=Y//\bG_m$ is the coarse moduli space of $[X]$.

Let us describe the category $\sO_{[X]}\Qc$ of quasicoherent sheaves on $[X]$.
Choose a homogeneous basis $\ve_i$ on $V$ 
with $\ve_i\in V_{d_i}$,
$i=0,1, \ldots, n$.
Let $\vx_i \in V^\ast$ be the dual basis. Then 
$\bK[V] = \bK [\vx_0,...,\vx_n ]$ possesses a natural grading
with  $\deg(\vx_{i})=d_{i}$.
Let $I$ be the defining ideal of $\overline{Y}$. Since $Y$ is $\bG_m$-invariant,
the ideal $I$ and the ring
$$
\bA\coloneqq \bK[\overline{Y}] = \bK [\vx_0,...,\vx_n ]/I
$$
are graded.
Both $X$ and $[X]$ can be thought of as the projective spectrum of
$\bA$.
The scheme $X$ is naturally isomorphic to the scheme theoretic 
$\Proj\, \bA$.
The stack $[X]$ is the Artin-Zhang projective spectrum
$\Proj_{AZ} \bA$ \cite{AKO}, 
i.e. its category of quasicoherent sheaves
$\sO_{[X]}\Qc$ is equivalent to the quotient category
$\bA\Gr/\bA\To$ where
$\bA\Gr$ is the category of
$\bZ$-graded $\bA$-modules, 
$\bA\To$ is its full subcategory of torsion modules.

Recall that 
$$
\tau_\bA (M) = \{m \in M \,\mid\, \exists N \; \forall k>N \; \bA_k m=0\}
$$
is {\em the torsion submodule of} $M$. $M$ is said to be {\em torsion} 
if $\tau_\bA (M)=M$. It can be seen as well that the torsion submodule of $M$ is the sum of all the finite dimensional submodules of $M$ since $\bA$ is connected.

Denote by
$$
\pi_\bA:\bA\Gr \rightarrow \bA\Gr/\bA\To
$$
the quotient functor. Since $\bA\Gr$ has enough injectives and $\bA\To$ is dense then there exists a section functor 
$$
\omega_\bA:\bA\Gr/\bA\To \rightarrow \bA\Gr
$$ which is right adjoint to $\pi_\bA$ in the sense that 
$$
\Hom_{\bA\mbox{\tiny \Gr}}(N,\omega_\bA(\mathcal{M}))\cong\Hom_{\bA\mbox{\tiny\Gr}/\bA\mbox{\tiny\To}}(\pi_\bA(N),\mathcal{M}).
$$
Recall that $\pi_\bA$ is exact, $\omega_\bA$ is left exact and $\pi_{\bA}\omega_{\bA} \cong Id_{\bA\mbox{\tiny\Gr}/\bA\mbox{\tiny\To}}$.
We call $\omega_\bA\pi_\bA(M)$ the {\em $\bA$-saturation} of $M$. 
We say that a module is {\em $\bA$-saturated} 
if it is isomorphic to the saturation of a module. It can be seen from the adjunction that an $\bA$-saturated module is torsion-free and is isomorphic to its own saturation. If $M$ and $N$ are $\bA$-saturated, then being isomorphic in $\bA\Gr/\bA\To$ is equivalent to being isomorphic in $\bA\Gr$.

We need a description of
the global sections functor on $[X]$ in these terms:
$$
\Gamma : \sO_{[X]}\Qc \rightarrow {\rm VS}_\bK, \ \ 
\Gamma (\mathcal{M}) = \omega_\bA(\mathcal{M})_0.
$$
In particular, if $M$ is an $\bA$-saturated module then $$
\Gamma(\pi_\bA(M))=M_0.
$$

The sheaf $\sO_{[X]}(k)$ is defined
as $\pi_\bA(\bA[k])$ where $\bA[k]$ is the shifted regular module and the grading is given by $\bA[k]_m = \bA_{k+m}$.

In particular, $\Gamma (\sO_{[X]}(k)) = \bA_k$ if $\bA[k]$ is $\bA$-saturated which is the case for polynomial rings of more than two variables \cite{AZ}. A well-known example of a ring, not $\bA$-saturated (as an $\bA$-module), is the polynomial ring in one variable $\bA=\bK[x]$. Its $\bA$-saturation is the Laurent polynomial ring $\bK[x,x^{-1}]$ seen as an $\bA$-module.
Finally we will need the push-forward functor
$$
\pi_\ast : \sO_{[X]}\Qc \rightarrow \sO_{X}\Qc ,
$$
given by associating a sheaf on $X$ to a graded $\bA$-module.
In general, it is not an equivalence. For instance,
$\sO_{[X]}(k)$ is an invertible sheaf but
$\sO_{X}(1)\cong \pi_\ast (\sO_{[X]}(1))$ is not invertible, in
general
\cite{Dol}.

Let us now describe the (twisted) $\sD_{[X]}$-modules.
Let $\partial_i = \partial / \partial \vx_i$, $i=0,1,\ldots, n$.
The Weyl algebra
$D(V)=\bK \langle \vx_0, \ldots ,\vx_n, \partial_0,
\ldots ,\partial_n \rangle$
gets a grading from the
$\bG_m$-action on $V$:
$\deg (\vx_i) = d_i$, 
$\deg (\partial_i) = -d_i$. 
We define {\em the reduced Weyl algebra} as
$$
\bD\coloneqq 
{\rm End}_{D(V)} (D(V)/ID(V)) \cong \bI(ID(V))/ID(V)$$
where 
$$
\bI(ID(V)) = \{ \vw \in D(V) \,\mid\, \vw I D(V) \subseteq I D(V) \}
$$ 
is the idealiser of $ID(V)$ in $D(V)$. Notice that $\bD$ is graded:
$I$ is graded, then $ID(V)$ is graded,
then $\bI(ID(V))$ is graded,
and finally $\bD$ is graded.
Observe that $\bA$ 
is a graded subalgebra of $\bD$ since $\bK[\vx_i] \subseteq \bI(ID(V))$.
It is known that for $\vw \in D(V)$ \cite[15.5.9]{MR}
$$
\vw \in ID(V) \Leftrightarrow \vw (\bK[\vx_i]) \subseteq I
\ \ \mbox{ and } \ \ 
\vw \in \bI(ID(V)) \Leftrightarrow \vw (I) \subseteq I
$$
where $\vw$ acts naturally on polynomials in $I$. This defines an algebra embedding
$\bD\hookrightarrow \End_\bK (\bA)$ whose image lies in
$D(\overline{Y})$,
the ring of differential operators on $\bA$. 
\begin{prop} \cite[15.5.13]{MR} \label{new_isom}
The map 
$\phi : \bD\rightarrow D(\overline{Y})$
is an isomorphism.
\end{prop}
The element $\sum_i d_i \vx_i \partial_i$
belongs to the idealiser $\bI(ID(V))$.
We call its image in $\bD$ {\em the Euler field}
$$
\vE=\sum_i d_i \vx_i \partial_i + ID(V).
$$
It belongs to $\bD_0$ and defines the grading of $\bD$ and its subalgebra
$\bA$.
\begin{lemma} 
\label{Efield}
Let $\vx \in \bD$. Then 
$\vx\in \bD_k$ if and only if 
$\vE \vx - \vx\vE = k\vx$.
\end{lemma}
\begin{proof}
It suffices to check it on the generators:
$$\vE\vx_i =\sum_j d_j\vx_j\partial_j \vx_i = \vx_i \vE + d_i \vx_i.$$
Similarly, 
$$\vE\partial_i =\partial_i \vE - d_i \partial_i.$$
\end{proof}
The Euler field can be used to define gradings on $\bD$-modules.
\begin{lemma} 
\label{Efield2}
Let $M$ be a $\bD$-module. The span $M^\prime$ of all
eigenvectors of the Euler field $\vE$
is a $\bK$-graded $\bD$-submodule of $M$.
\end{lemma}
\begin{proof}
Let $m\in M^\lambda$, the $\lambda$-eigenspace of $\vE$.
Using Lemma~\ref{Efield},
$$\vE\vx_i m= \vx_i \vE m + d_i \vx_im = (\lambda +d_i)\vx_i m,$$
so  $$\vx_im\in M^{\lambda+d_i}.$$
Similarly, 
$$\vE\partial_i m=\partial_i \vE m - d_i \partial_i m = 
(\lambda -d_i)\partial_i m$$
and  $$\partial_im\in M^{\lambda-d_i}.$$
\end{proof}
Let us fix $\lambda\in\bK$.
In general, 
$$M\geq M^\prime = \oplus_{\mu\in\bK} M^\mu
\geq M^{(\lambda)} := \oplus_{n\in\bZ} M^{\lambda+n}.$$
A $\bD$-module $M$ is called
{\em $\lambda$-Euler} if $M=M^{(\lambda)}$.
A $\lambda$-Euler $\bD$-module $M$ admits a canonical
$\bZ$-grading given by $M_k=M^{k+\lambda}$. {\em The category of $\lambda$-Euler $\bD$-modules} $\bD\Gr^\lambda$ is a full subcategory of the category of graded $\bD$-modules $\bD\Gr$.
The full subcategory of the torsion (as $\bA$-modules) modules
is denoted $\bD\To^\lambda$. Notice as well that the torsion submodule 
of a graded $\bD$-module is a graded $\bD$-module and that if, moreover, it is $\lambda$-Euler, then the torsion submodule is $\lambda$-Euler too.

$\bD\Gr^\lambda$ is a locally small category. $\bD\To^\lambda$ is a Serre subcategory of $\bD\Gr^\lambda$ which is closed under taking arbitrary direct sums. Therefore, $\bD\To^\lambda$ is a localising subcategory of $\bD\Gr^\lambda$ \cite{Ga} and the quotient functor 
$$
\pi^\lambda_\bD:\bD\Gr^\lambda \rightarrow \bD\Gr^\lambda / \bD\To^\lambda
$$
is exact and has a right adjoint section functor
$$
\omega^\lambda_\bD:\bD\Gr^\lambda / \bD\To^\lambda \rightarrow \bD\Gr^\lambda.
$$
It follows that we have 
$$
\Hom_{\bD\mbox{\tiny\Gr}^\lambda}(N,\omega^\lambda_\bD(\mathcal{M}))\cong\Hom_{\bD\mbox{\tiny\Gr}^\lambda / \bD\mbox{\tiny\To}^\lambda}(\pi^\lambda_\bD(N),\mathcal{M}).
$$

\begin{theorem}
\label{DMod-q}
The category
$\sD_{[X]}\Qc$
of quasicoherent D-modules on the stack $[X]$ is equivalent
to the quotient category
$\bD\Gr^0 / \bD\To^0$.
\end{theorem}
\begin{proof}
The category of D-modules on $\overline{Y}$ is just the category of
$D(\overline{Y}$)-modules 
since $\overline{Y}$ is affine.
The category of weakly $\bG_m$-equivariant D-modules on $\overline{Y}$
is $D(\overline{Y})\Gr$. The two actions of the Lie algebra of the
multiplicative group $\bG_m$ are given by the Euler element $\vE$ and by 
the grading. Thus, the category of strongly $\bG_m$-equivariant
D-modules on 
$\overline{Y}$
is the category of 0-Euler D-modules $D(\overline{Y})\Gr^0$.

By definition, the category 
$\sD_{[X]}\Qc$ is the category of strongly $\bG_m$-equivariant
D-modules on $Y$. Thus, taking sections on the open set $Y$ induces an exact functor
$$
\Gamma (Y, \underline{\hspace{3mm}}\;) : 
\sD_{[X]}\Qc \rightarrow 
D(Y)\Gr
$$
where $D(Y)$ is the ring of global differential operators
on $Y$. 
%
Proposition~\ref{new_isom} makes the global sections
$\Gamma (Y, \sM)$ into a graded $\bD$-module via the restriction map
$\bD\cong D(\overline{Y})\rightarrow D(Y)$.
This module is 0-Euler, because $\sM$ is strongly equivariant.
Thus, we obtain exact functors
$$
\Gamma (Y, \underline{\hspace{3mm}}\;):
\sD_{[X]}\Qc \rightarrow \bD\Gr^0
\ \ \ \mbox{ and }
$$
$$
\pi_\bD^0 \circ \Gamma (Y, \underline{\hspace{3mm}}\;):
\sD_{[X]}\Qc \rightarrow \bD\Gr^0/\bD\To^0.
$$
Let us examine the sheafification functor $\bD\Gr^0 \rightarrow \sD_{[X]}\Qc$.
The sheafification of an object in $\bD\To^0$ is supported at $0$. Hence 
objects in $\bD\To^0$ give the zero sheaf on $Y$. So it induces a functor on the quotient
$$
\widetilde{\;}: \bD\Gr^0/\bD\To^0 \rightarrow \sD_{[X]}\Qc
$$
which is quasiinverse to $
\pi_\bD^0 \circ \Gamma (Y, \underline{\hspace{3mm}}\;)$.
\end{proof}

An inquisitive reader may observe that
we have defined the category $\sD_{[X]}\Qc$
without defining the object $\sD_{[X]}$.
Later on we remedy this partially by constructing
an object $D_{[X]}^\lambda$
for each $\lambda \in \bK$
so that $\sD_{[X]} = \pi^0_\bD(D_{[X]}^0)$. 
Let us define 
{\em the category 
$\sD_{[X]}^\lambda\Qc$
of twisted D-modules on $[X]$}
as the quotient
$\bD\Gr^\lambda / \bD\To^\lambda$. It is possible to define the category
internally and then prove
a version of Theorem~\ref{DMod-q} but we see no value in doing it here.

Given a module $M$ in $\bD\Gr^\lambda$, we call $\omega^\lambda_\bD\pi^\lambda_\bD(M)$ the {\em $\bD^\lambda$-saturation} of $M$. We say that a module is {\em $\bD^\lambda$-saturated} is it is isomorphic to the $\bD^\lambda$-saturation of a module. It can be seen from the adjunction that a $\bD^\lambda$-saturated module is torsion-free and is isomorphic to its own saturation.

We shall prove now that an $\bA$-saturated $\lambda$-Euler $\bD$-module is automatically $\bD^\lambda$-saturated. This will make our forthcoming calculations easier.

\begin{lemma}
\label{}
Let $M$ be a $\lambda$-Euler $\bD$-module. 
Then the $\bD^\lambda$-saturation of $M$ is an $\bA$-submodule of its $\bA$-saturation.
\end{lemma}

\begin{proof}
	We have a map $$M\rightarrow
        \omega^\lambda_\bD\pi^\lambda_\bD(M)$$ in $\bD\Gr^\lambda$
        \cite{AZ}. The kernel and cokernel of this map are torsion
which implies that 
	$$
	\pi_\bA(\omega^\lambda_\bD\pi^\lambda_\bD(M)) \cong \pi_\bA(M).
	$$
	From adjunction, this isomorphism is the image of a map in $\bA\Gr$,
	$$
	\phi:\omega^\lambda_\bD\pi^\lambda_\bD(M) \rightarrow \omega_\bA\pi_\bA(M).
	$$
    We claim that this map is injective. Since $\pi_\bA(\phi)$ is an isomorphism then $\Ke\phi$ is a torsion $\bA$-module. Consider $\bD\Ke\phi$ (which contains $\Ke\phi$), it is a left $\bD$-submodule of $\omega^\lambda_\bD\pi^\lambda_\bD(M)$. Take $m\in \Ke\phi$ then there exists an integer $N$ such that 
    $$
    \bA_{\geqslant N}m=0.
    $$
    For any $d\in \bD
    $
    of order $k$ we have 
    $$
    \bA_{\geqslant N+k}(dm)\leqslant \bD \bA_{\geqslant N}m=0.
    $$
    It follows that it is a torsion submodule of $\omega^\lambda_\bD\pi^\lambda_\bD(M)$ but $\omega^\lambda_\bD\pi^\lambda_\bD(M)$ is torsion-free. Hence $\Ke\phi=0$
\end{proof}

An immediate corollary is the following:
\begin{cor}
	Any $\bA$-saturated $\lambda$-Euler $\bD$-module is $\bD^\lambda$-saturated.
\end{cor}

Let us give examples of objects in $\sD_{[X]}^\lambda\Qc$.
The sheaf $\sO_{[X]}(k)$ is an object in $\sD_{[X]}^k\Qc$. 
We introduce
$$
D_{[X]}^\lambda := \bD/\bD(\vE-\lambda).
$$
Another interesting object in $\sD_{[X]}^\lambda\Qc$ is
$$
\sD_{[X]}^\lambda := \pi^\lambda_\bD(D_{[X]}^\lambda).
$$
It plays the role of the sheaf of twisted differential operators,
although $D_{[X]}^\lambda$ is not an algebra because $\bD(\vE-\lambda)$
is not a two-sided ideal, in general.
However,  $\vE$ is a central element of $\bD_0$,
so
$$
{D_{[X]}^\lambda}_0 = \bD_0/\bD_0(\vE-\lambda)
$$
is an algebra. It
plays the role of the algebra of global sections
of the twisted differential operators on $[X]$.
${D_{[X]}^\lambda}$ is a $\bD-{D_{[X]}^\lambda}_0$-bimodule.

In the next section the adjoint functors
of global sections and localisation will play a role.
This adjoint pair  
$( \Gamma_\lambda , L_\lambda)$
is defined as:
$$ \Gamma_\lambda : \sD_{[X]}^\lambda\Qc \rightarrow {D_{[X]}^\lambda}_0\Mo, \ 
\Gamma_\lambda (\mathcal{M}) := \omega^\lambda_\bD(\mathcal{M})_0 = \omega^\lambda_\bD(\mathcal{M})^\lambda,
$$
$$ L_\lambda : {D_{[X]}^\lambda}_0\Mo \rightarrow \sD_{[X]}^\lambda\Qc, \ \ 
L_\lambda (N) := \pi^\lambda_\bD(D_{[X]}^\lambda\otimes_{{D_{[X]}^\lambda}_0} N).
$$

The ways we defined our global sections functors 
for $\sD_{[X]}^\lambda\Qc$ and $\sO_{[X]}\Qc$ are not necessarily
equivalent. 
Yet we know that 
$$
\Gamma_\lambda(\pi^\lambda_\bD(M)) \leqslant \Gamma(\pi_\bA(M))
$$
as $\bA$-modules for any $\lambda$-Euler $\bD$-module $M$.

The exposition would be greatly simplified if restricting the section functor $\omega_\bA$ to $\sD_{[X]}^\lambda\Qc$ were equivalent to $\omega^\lambda_\bD$. This explains why we have different global sections functor for different $\lambda$ although geometrically only one is needed. However, to ensure that we obtain $\lambda$-Euler $\bD$-modules and not just $\bA$-modules we use $\omega^\lambda_\bD$.

\section{D-modules on weighted projective space}

In this section we consider
$Y= V\setminus\{ 0\}$, the punctured vector space  
of dimension at least 2
and
$[X]=[Y/\bG_m]=[\bP(V)]$, the weighted projective 
stack.
In this case $I=\{0\}$,
$\bA=\bK[\vx_0, \ldots, \vx_n]$ where the degree of $\vx_i$ is $d_i>0$
and 
$\bD=\bK\langle \vx_0, \ldots, \vx_n ,\partial_0, \ldots, \partial_n\rangle$
is the Weyl algebra. 
Without loss of generality, we assume that
$0<d_0 \leq d_1 \leq \ldots \leq d_n$.

Let us look at the $\bD$-module $\Delta$ generated by the
delta-function at zero $\delta=\delta_0(\vx_0, \ldots, \vx_n )$
$$
\Delta
= \bD \delta
\cong \bD /
( \bD\vx_0 + \bD\vx_1 + \ldots + \bD\vx_n )\, . 
$$
The linear map 
$$
\bK
[\partial_0, \ldots, \partial_n ]
\rightarrow
\Delta, \ \ 
f (\partial_0, \ldots, \partial_n )
\mapsto
f (\partial_0, \ldots, \partial_n ) \cdot \delta
$$
is an isomorphism of 
vector spaces.
If we identify
$\bK
[\partial_0, \ldots, \partial_n ]
$
with
$
\Delta
$
using this linear map, then
$\partial_i$ acts by multiplication
and
$\vx_i$ acts by derivation
$\partial_j \mapsto - \delta_{i,j}$.
In particular, 
$$
\vE\cdot \delta=
\vE\cdot 1 =
\sum_j d_j\vx_j\cdot \partial_j
=\sum_j -d_j = 
-(\sum_j d_j)\delta
.$$
Hence, 
$\Delta$ is $k$-Euler for each integer $k$.
Its canonical $k$-Euler
grading is given by 
$$
\delta \in \Delta^{-\sum_j d_j} = \Delta_{-k - \sum_j d_j}, \ \ 
\partial_i\cdot\delta \in \Delta_{-k - d_i - \sum_j d_j}
.$$

Let $J=(\vx_0, \ldots, \vx_n) \lhd \bA$.
If $M$ is a $\bD$-module,
$\tau_\bA (M) = \{m\in M \,\mid\, \exists k \; J^km=0\}$ 
is its torsion $\bD$-submodule
(a reader can easily verify that if $J^km=0$, 
then $J^{k+1}\partial_im=0$).
The torsion $\bD$-modules are those, supported set theoretically on the
zero $0\in V$.
By Kashiwara's theorem, any $\bD$-module supported at $0$ 
is a direct sum of copies of $\Delta$.

Let us introduce some notations. Suppose that $M$ and $N$ are two $\bZ$-graded $\bA$-modules. We say that an $\bA$-module homomorphism $f:M \rightarrow N$ has \emph{degree} $l$ if $f(M_i) \subset N_{i+l}$ for all $i$. Denote by $\Hom(M,N)_l$ the set of all degree $l$ $\bA$-module homomorphisms and write
$$
\underline{\Hom}_\bA(M,N)=\bigoplus_{l\in \bZ} \Hom(M,N)_l.
$$
Now let $\Ext^q(M,N)_l$ be the derived functor of $\Hom(M,N)_l$ and write
$$
\underline{\Ext}^q_\bA(M,N)=\bigoplus_{l\in \bZ} \Ext^q(M,N)_l.
$$
Artin and Zhang prove \cite{AZ} that for any graded $\bA$-module $M$,
\begin{align*}
\tau_\bA(M) &\cong \underrightarrow{\lim} \; \underline{\Hom}_\bA(\bA/\bA_{\geqslant k},M),\\
R^{1}\tau_\bA(M) &\cong \underrightarrow{\lim}\; \underline{\Ext}^1_\bA(\bA/\bA_{\geqslant k},M)
\end{align*}
and that there exists a long exact sequence of $\bA$-modules
$$
	0  \rightarrow \tau_\bA(M) \rightarrow M \rightarrow \omega_\bA\pi_\bA(M) \rightarrow R^{1}\tau_\bA(M) \rightarrow 0
$$
where $\tau_\bA(M)$ and $R^{1}\tau_\bA(M)$ are torsion modules. This implies the following proposition.
\begin{prop}
	\label{Dsat}
	A $\lambda$-Euler $\bD$-module M is $\bD^\lambda$-saturated if it is torsion-free and $\underrightarrow{\lim} \; \underline{\Ext}^1(\bA/\bA_{\geqslant k},M) = 0$.
\end{prop}

The next lemma will prove primordial in the proof that $\Gamma_\lambda L_\lambda \cong Id_{{D_{[X]}^\lambda}_0\mbox{\tiny{\Mo}}}$ for any $\lambda$ and $n\geqslant 2$. 

\begin{lemma}
	For $n \geqslant 2$, $D_{[X]}^\lambda$ is $\bD^\lambda$-saturated.
\end{lemma}

\begin{proof}
	Recall that $D_{[X]}^\lambda = \bD/\bD(\vE-\lambda)$. It is easier to compute Ext groups by taking a projective resolution of the left argument than an injective one of the right argument. Since $\bA/\bA_{\geqslant 1} \cong \bK$, the first three terms of the Koszul resolution are given by
	$$
	\ldots \rightarrow \bigoplus_{i_0 < i_1} \bA(-d_{i_0}-d_{i_1}) \rightarrow \bigoplus^n_{i=0} \bA(-d_i) \rightarrow \bA \rightarrow \bA/\bA_{\geqslant 1} \rightarrow 0.
	$$
	Take away $\bA/\bA_{\geqslant 1}$ and apply $\underline{\Hom}_\bA(\underline{\hspace{3mm}},D_{[X]}^\lambda)$ to the above exact sequence to get
	$$
	0 \rightarrow D_{[X]}^\lambda \overset{\phi_1}{\rightarrow} \bigoplus^n_{i=0} D_{[X]}^\lambda(d_i) \overset{\phi_2}{\rightarrow} \bigoplus_{i_0 < i_1} D_{[X]}^\lambda(d_{i_0}+d_{i_1}) \rightarrow  \ldots
	$$
	where 
	$$
	\phi_1\colon \overline m \mapsto (\vx_i\overline m)^n_{i=0}
	$$
	and
	$$
	\phi_2\colon (\overline m_i )_{i=0}^n \mapsto (\vx_{i_0}\overline m_{i_1}-\vx_{i_1}\overline m_{i_0})_{i_0<i_1}.
	$$
	It follows that
	\begin{align*}
	\underline{\Hom}_\bA(\bA/\bA_{\geqslant 1},D_{[X]}^\lambda) &\cong \Ke(\phi_1),\\
	\underline{\Ext}^1_\bA(\bA/\bA_{\geqslant 1},D_{[X]}^\lambda) &\cong \frac{\Ke(\phi_2)}{\Ima(\phi_1)}.
	\end{align*}
	Both $\underline{\Hom}_\bA(\bA/\bA_{\geqslant 1},D_{[X]}^\lambda)$ and $\underline{\Ext}^1_\bA(\bA/\bA_{\geqslant 1},D_{[X]}^\lambda)$ vanish.
	Let us first compute $\underline{\Hom}_\bA(\bA/\bA_{\geqslant 1},D_{[X]}^\lambda)$. Pick $\overline m \in \Ke(\phi_1)$, then $\vx_i\overline m=0$ for each $i$, where $$\overline m=m+\bD(\vE-\lambda).$$ We can assume $m$ to be homogeneous, so $$\vx_im=p_i(\vE-\lambda)$$ for some homogeneous $p_i \in \bD$. We want to show that $p_i\in \vx_i\bD$. Suppose, for a contradiction, that it is not. Then we can write
	$$
	p_i=\vx_i m' + \vf\partial^{\underline{\beta}} + LT
	$$
	where $m'\in \bD$, $\vf\in \bK[\vx_0,\ldots,\vx_n]$ is the
	highest term which is non-zero by assumption, free of $\vx_i$, $\underline{\beta}$ the biggest
	power and $LT$ are the lower terms using 
	{\bf DegLex} for the ordering of the monomials in $\partial$.
	Without loss of generality, we can assume that $i\neq0$. It follows that
	$$
	\vx_im =\vx_i m^{\prime\prime}+d_0\vf\vx_0\partial^{\underline{\beta}+\underline{e_0}} + LT
	$$ 
	since $\vf\partial^{\underline{\beta}}\vx_0\partial_0=\vf\vx_0\partial^{\underline{\beta}+\underline{e_0}} + LT$. But $\vf\vx_0$ is not divisible by $\vx_i$ and we obtain a contradiction. Thus, $$\underline{\Hom}_\bA(\bA/\bA_{\geqslant 1},D_{[X]}^\lambda)=0.
	$$
	
	Similarly, let us show that $\underline{\Ext}^1_\bA(\bA/\bA_{\geqslant 1},\bD_{[X]}^\lambda)$ vanishes. To proceed, choose $(\overline m_i)^n_{i=0}\in \Ke(\phi_2)$. Then for all $i,j$, there exists a $\theta_{ij}\in \bD$ such that
	$$
	\vx_im_j = \vx_jm_i + \theta_{ij}(\vE-\lambda).
	$$
	Write $$
	m_j=\vx_jm'_j + \vf\partial^{\underline{\beta}} + LT
	$$
	where $m'_j\in \bD$, $\vf\in \bK[\vx_0,\ldots,\vx_n]$ is the
	highest  term, free of $\vx_j$, $\underline{\beta}$ is the
	highest power and $LT$ are the lower terms using 
	{\bf DegLex} for the ordering of the monomials in $\partial$. Let us suppose, for the sake of a contradiction, that $|\underline{\beta}| \neq 0$. Then without loss of generality, we can assume that $\underline{\beta}$ is the lowest among all the possible representatives of $\overline m_j$.
	Write 
	$$
	\theta_{ij}=\vx_j\theta'+\vg \partial^{\underline{\gamma}} + LT
	$$
	where $\vg\in \bK[\vx_0,\ldots,\vx_n]$ is the highest term, 
	free of $\vx_j$. If $\vg=0$ then we are done. Suppose that $\vg \neq 0$ so that
	$$
	\vx_i\vx_jm'_j + \vx_i\vf\partial^{\underline{\beta}} + LT = \vx_j(m_i + \theta'(\vE-\lambda)) + \vg \partial^{\underline{\gamma}}(\vE-\lambda) + LT.
	$$
	Again without loss of generality, suppose that $i,j \neq 0$ as $n \geqslant 2$. By comparing the highest terms, free of $\vx_j$, we get
	$$
	\vx_i\vf\partial^{\underline{\beta}} = d_0\vg\vx_0\partial^{\underline{\gamma}+\underline{e_0}}
	$$
	with $|\underline \gamma| < |\underline \beta|$. Therefore,
	$$
	\vf\partial^{\underline{\beta}} = d_0\frac{\vg}{\vx_i}\vx_0\partial^{\underline{\gamma}+\underline{e_0}} \\
	= \frac{\vg}{\vx_i}\partial^{\underline{\gamma}}(\vE-\lambda) + LT.
	$$
	So $m_j-\frac{\vg}{\vx_i}\partial^{\underline{\gamma}}(\vE-\lambda)$ is another representative of $\overline m_j$ which has an index $\underline{\gamma}$ lower than $\underline{\beta}$, contrary to our hypothesis. Thus $\vg =0$ and
	$$
	m_j=\vx_jm'_j
	$$
	For all $i,j$, we have
	$$
	\vx_i\vx_jm'_j=\vx_i\vx_jm'_i + \theta_{ij}(\vE-\lambda)
	$$
	which implies that
	$$
	\vx_i\vx_j(m'_j-m'_i) \in \bD(\vE-\lambda).
	$$
	By using the first argument twice, we obtain that for all $i,j$
	$$
	m'_j-m'_i\in \bD(\vE-\lambda).
	$$
	Write
	$$
	\overline{m'} \coloneqq \overline{m'_j}=\overline{m'_i}
	$$
	for the residues of $m'_j$ and $m'_i$. Then for all $i$,
	$$
	\overline{m_i} = \vx_i\overline{m'}.
	$$
	Hence,
	$$\underline{\Ext}^1_\bA(\bA/\bA_{\geqslant 1},D_{[X]}^\lambda)=0.
	$$
	
	To finish our proof, for each $k$ we have a short exact sequence of graded $\bA$-modules:
	$$
	0  \rightarrow \bA_{\geqslant k}/\bA_{\geqslant k+1} \rightarrow \bA/\bA_{\geqslant k+1} \rightarrow \bA/\bA_{\geqslant k} \rightarrow  0
	$$
	and $\bA_{\geqslant k}/\bA_{\geqslant k+1}$ is isomorphic to a finite direct sum of copies of $\bA/\bA_{\geqslant 1}$.
	By applying $\underline{\Hom}_\bA(\underline{\hspace{3mm}},D_{[X]}^\lambda)$ to this short exact sequence and by induction on $k$, we conclude that for all $k$:
	\begin{align*}
	\underline{\Hom}_\bA(\bA/\bA_{\geqslant k},D_{[X]}^\lambda)&=0,\\
	\underline{\Ext}^1_\bA(\bA/\bA_{\geqslant k},D_{[X]}^\lambda)&=0.
	\end{align*}
	Taking direct limit \cite{AZ} it follows that 
	$$
	\tau_\bA(D_{[X]}^\lambda) = 0,
	\ \ \mbox{ and } \ \ 
	\underrightarrow{\lim} \; \underline{\Ext}^1(\bA/\bA_{\geqslant k},D_{[X]}^\lambda) = 0.
	$$
	Hence $D_{[X]}^\lambda$ is $\bD^\lambda$-saturated by Proposition~\ref{Dsat}.
\end{proof}
The condition on $n$ in the last proof is necessary. We can prove that $D_{[X]}^\lambda$ is not $\bD^\lambda$-saturated for all $\lambda$ when $n=1$ . For this, it suffices to notice that for $\lambda=0$,
$$
(-d_1\partial_1,d_0\partial_0)\in \Ke(\phi_2)
$$
but
$$
(-d_1\partial_1,d_0\partial_0)\notin \Ima(\phi_1)
$$
since $d_0\vx_0\partial_0=-d_1\vx_1\partial_1+\vE$.
\begin{lemma}
	 Let $n\geqslant2$. If $\Gamma_\lambda$ is exact then $\Gamma_\lambda L_\lambda \cong Id_{{D_{[X]}^\lambda}_0\mbox{\tiny{\Mo}}}$
\end{lemma}

\begin{proof}
	Let $N$ be a ${D_{[X]}^\lambda}_0$-module. Take the first two terms of a free resolution of $N$
	$$
		P_1 \rightarrow P_0 \rightarrow N \rightarrow 0
	$$
	where $P_i=\bigoplus \limits_{j\in I_i} {D_{[X]}^\lambda}_0$ and $I_i$ is an index set. Since both ${D_{[X]}^\lambda} \otimes_{{D_{[X]}^\lambda}_0} \underline{\hspace{3mm}}$ and $\pi^\lambda_\bD$ are right exact functors, it follows that
	$$
		 \Gamma_\lambda L_\lambda(P_1) \rightarrow \Gamma_\lambda L_\lambda(P_0) \rightarrow \Gamma_\lambda L_\lambda(N) \rightarrow 0
	$$
	is exact. We can compute the first two terms explicitly:
	\begin{align*}
	\Gamma_\lambda L_\lambda(P_i) &=(\omega^\lambda_\bD\pi^\lambda_\bD({D_{[X]}^\lambda} \otimes_{{D_{[X]}^\lambda}_0} P_i))_0\\ &=(\omega^\lambda_\bD\pi^\lambda_\bD({D_{[X]}^\lambda} \otimes_{{D_{[X]}^\lambda}_0} \bigoplus \limits_{j\in I_i} {D_{[X]}^\lambda}_0))_0 \\
	&\cong(\omega^\lambda_\bD\pi^\lambda_\bD(\bigoplus \limits_{j\in I_i}{D_{[X]}^\lambda} \otimes_{{D_{[X]}^\lambda}_0} {D_{[X]}^\lambda}_0))_0 \\
	&\cong (\omega^\lambda_\bD\pi^\lambda_\bD(\bigoplus \limits_{j\in I_i} {D_{[X]}^\lambda}))_0
	\end{align*}
	since the tensor product commutes with arbitrary direct sums and that ${D_{[X]}^\lambda} \otimes_{{D_{[X]}^\lambda}_0} {D_{[X]}^\lambda}_0 \cong {D_{[X]}^\lambda}$.
	The category  $\bD\Gr^\lambda$ is locally noetherian \cite[Prop. 4.18]{El}. By a result of Gabriel, the section functor $\omega^\lambda_\bD$ commutes with inductive limits and, in particular, with arbitrary direct sums \cite[p. 379]{Ga}. Moreover, $\pi^\lambda_\bD$ is left adjoint to $\omega^\lambda_\bD$, so $\pi^\lambda_\bD$ commutes as well with arbitrary direct sums. This yields the following sequence of natural isomorphisms:
	\begin{align*}
		\Gamma_\lambda L_\lambda(P_i)
		&\cong (\omega^\lambda_\bD\pi^\lambda_\bD(\bigoplus \limits_{j\in I_i} {D_{[X]}^\lambda}))_0 \\
		&\cong (\bigoplus \limits_{j\in I_i} {\omega^\lambda_\bD\pi^\lambda_\bD (D_{[X]}^\lambda}))_0 \\
		&\cong (\bigoplus \limits_{j\in I_i} {D_{[X]}^\lambda})_0 \\
		&\cong \bigoplus \limits_{j\in I_i} {D_{[X]}^\lambda}_0 \\
		&\cong P_i
	\end{align*}
	since ${D_{[X]}^\lambda}$ is $\bD^\lambda$-saturated and that $(\underline{\hspace{3mm}})_0$ commutes with arbitrary direct sums. 
	Thus, we constructed a commutative diagram with exact rows:
	$$
	\xymatrix{
	&P_1 \ar[r] \ar[d]^\alpha &P_0 \ar[r] \ar[d]^\beta &\Gamma_\lambda L_\lambda(N) \ar[r] \ar[d]^\gamma &0 \ar[d] \\
	&P_1 \ar[r] &P_0 \ar[r] &N \ar[r] &0
	}
	$$
	where $\alpha$ and $\beta$ are isomorphisms, so $\Gamma_\lambda L_\lambda(N) \cong N$ is a natural isomorphism by the four lemma.
\end{proof}

\begin{theorem}
\label{DMod-eq}
Let $\sA$ be the $\bZ_{\geq 0}$-span of all $d_i$-s.
If $\lambda \in \bK \setminus (-\sum_i d_i -\sA)$, 
then the global sections functor
$\Gamma_\lambda : \sD_{[X]}^\lambda \Qc \rightarrow {D_{[X]}^\lambda}_0\Mo$
is exact.
In this case, $\Gamma_\lambda$ defines an equivalence between 
the quotient category
$\sD_{[X]}^\lambda \Qc / \Ke \Gamma_\lambda$ and ${D_{[X]}^\lambda}_0\Mo$.
\end{theorem}
\begin{proof}
The category  $\sD_{[X]}^\lambda\Qc$
is the quotient category of the category of $\lambda$-Euler modules
by the category of torsion modules. 
The canonical grading on a $\lambda$-Euler module
$M$ is given by
$M_k = M^{k+\lambda}$. 
The torsion modules are direct sums of
$\Delta$.
%
The global sections functor $\Gamma_\lambda$
is 
$$
\Gamma_\lambda : \mathcal{M} \mapsto 
\omega^\lambda_\bD(\mathcal{M})_0 = 
\omega^\lambda_\bD(\mathcal{M})^\lambda .
$$

We know that $\omega^\lambda_\bD$ is a left exact functor. Taking $\lambda$-eigenspaces is an exact functor, so we are left to prove that $\Gamma_\lambda$ is right exact. An epimorphism $f:\mathcal{M}\rightarrow \mathcal{N}$ induces the exact sequence
$$
\omega^\lambda_\bD(\mathcal{M}) \rightarrow \omega^\lambda_\bD(\mathcal{N}) \rightarrow \coke(\omega^\lambda_\bD(f)) \rightarrow 0
$$
where $\coke(\omega^\lambda_\bD(f))$ 
is a torsion $\bD$-module. Taking the zeroeth graded part, we get the exact sequence
$$
\Gamma_\lambda(\mathcal{M}) \rightarrow \Gamma_\lambda(\mathcal{N}) \rightarrow \coke(\omega^\lambda_\bD(f))_0 \rightarrow 0.
$$

Our restriction on $\lambda$ provides that  $\coke(\omega^\lambda_\bD(f))_0=0$.
Indeed, 
if $\lambda\not\in\bZ$, 
then $\coke(\omega^\lambda_\bD(f)) =0$. 
If $\lambda\in\bZ$, 
then $\coke(\omega^\lambda_\bD(f)) = \oplus \Delta$
and
$\coke(\omega^\lambda_\bD(f))_0 = 
\oplus \Delta^\lambda$.
Since
the $\vE$-weights of
$\Delta$
are $-\sum_i d_i-\sA$,
$\coke(\omega^\lambda_\bD(f))_0 = 0$.
Hence $\Gamma_\lambda$ is exact.

The kernel $\Ke \Gamma_\lambda$ is the full subcategory of
$\sD_{[X]}^\lambda \Qc$ whose objects are those $\mathcal{M}$ without non-trivial global
sections,
i.e., with $\Gamma_\lambda (\mathcal{M}) =0$.
Since $\Gamma_\lambda$ is exact, it is a Serre subcategory,
and $\Gamma_\lambda$ descends to a functor
$$
\widetilde{\Gamma}_\lambda : \sD_{[X]}^\lambda \Qc / \Ke \Gamma_\lambda
\rightarrow {D_{[X]}^\lambda}_0\Mo.
$$
and let 
$$
Q:\sD_{[X]}^{\lambda}\Qc \rightarrow \sD_{[X]}^\lambda \Qc / \Ke \Gamma_\lambda
$$
be the quotient functor. We claim that $QL_\lambda$ is a
quasiinverse of $\widetilde{\Gamma}_\lambda$. 
Now in one direction,
\begin{align*}
\widetilde{\Gamma}_\lambda(QL_\lambda)(N)
&=(\widetilde{\Gamma}_\lambda Q)L_\lambda (N) \\
&=\Gamma_\lambda L_\lambda (N) \\
&\cong N
\end{align*} 
since $\Gamma_\lambda$ is exact.
Thus, 
$$
\widetilde{\Gamma}_\lambda QL_\lambda \cong Id_{{D_{[X]}^\lambda}_0 \mbox{\tiny\Mo}}.
$$

In the opposite direction, we have a natural transformation
$$
QL_\lambda \widetilde{\Gamma}_\lambda \rightarrow Id_{\sD_{[X]}^\lambda \mbox{\tiny\Qc} /
  \mbox{\tiny\Ke} \Gamma_\lambda}.
$$
Take an object $\mathcal{\widetilde{M}}$ in ${\sD_{[X]}^\lambda \Qc /
\Ke \Gamma_\lambda}$. Then there exists an object $\mathcal{M}$ in $\sD_{[X]}^\lambda \Qc$ such that $\mathcal{\widetilde{M}}=Q(\mathcal{M})$. Hence,
\begin{align*}
QL_\lambda\widetilde{\Gamma}_\lambda(\mathcal{\widetilde{M}})
&= QL_\lambda\Gamma_\lambda(\mathcal{M})\\
&= Q\pi^\lambda_\bD({D_{[X]}^\lambda}\otimes_{{D_{[X]}^\lambda}_0} (\omega^\lambda_\bD(\mathcal{M}))_0).
\end{align*}
On a level of a $\lambda$-Euler module $M$ (with its canonical grading), 
the natural map
$$
{D_{[X]}^\lambda}\otimes_{{D_{[X]}^\lambda}_0} M_0 \rightarrow M
$$
gives rise to the long exact sequence
$$
0 \rightarrow K \rightarrow 
{D_{[X]}^\lambda}\otimes_{{D_{[X]}^\lambda}_0} M_0 \rightarrow M
\rightarrow N \rightarrow 0
$$
where $K$ is its kernel and $N$ is its cokernel.
Since $\pi^\lambda_\bD$ is exact,
$$
0 \rightarrow \pi^\lambda_\bD(K) \rightarrow 
\pi^\lambda_\bD({D_{[X]}^\lambda}\otimes_{{D_{[X]}^\lambda}_0} M_0) \rightarrow \pi^\lambda_\bD(M)
\rightarrow \pi^\lambda_\bD(N) \rightarrow 0
$$
is a long exact sequence as well.
If $M=\omega^\lambda_\bD(\mathcal{M})$, applying $\Gamma_\lambda$ yields
$$
0 \rightarrow \Gamma_\lambda\pi^\lambda_\bD(K) \rightarrow 
\omega^\lambda_\bD(\mathcal{M})_0 \rightarrow \omega^\lambda_\bD(\mathcal{M})_0
\rightarrow \Gamma_\lambda\pi^\lambda_\bD(N) \rightarrow 0
$$
since $\Gamma_\lambda\pi^\lambda_\bD(\omega^\lambda_\bD(\mathcal{M})) \cong \omega^\lambda_\bD(\mathcal{M})_0$ and $\Gamma_\lambda L_\lambda \cong Id_{{D_{[X]}^\lambda}_0 \mbox{\tiny\Mo}}$ when $\Gamma_\lambda$ is exact. The middle map $$\omega^\lambda_\bD(\mathcal{M})_0 \rightarrow \omega^\lambda_\bD(\mathcal{M})_0$$ is the identity map and hence an isomorphism. It follows that $\pi^\lambda_\bD(K)$ and $\pi^\lambda_\bD(N)$ are objects in $\Ke(\Gamma_\lambda)$. Therefore,
$$
\pi^\lambda_\bD({D_{[X]}^\lambda}\otimes_{{D_{[X]}^\lambda}_0} \omega^\lambda_\bD(\mathcal{M})_0) \rightarrow \pi^\lambda_\bD(\omega^\lambda_\bD(\mathcal{M}))
$$
is an isomorphism in ${\sD_{[X]}^\lambda \Qc / \Ke \Gamma_\lambda}$
and
\begin{align*}
QL_\lambda\widetilde{\Gamma}_\lambda(\mathcal{\widetilde{M}})
&\cong Q\pi^\lambda_\bD(\omega^\lambda_\bD(\mathcal{M})) \\
&\cong Q(\mathcal{M}) \\
&\cong \mathcal{\widetilde{M}}.
\end{align*}
It follows that 
$QL_\lambda \widetilde{\Gamma}_\lambda \cong I_{\sD_{[X]}^\lambda \mbox{\tiny\Qc} /
  \mbox{\tiny\Ke} \Gamma_\lambda}$.
\end{proof}

We are left to study when $\Ke \Gamma_\lambda$ is a zero category so that $\Gamma_\lambda$ defines an equivalence between 
the quotient category
$\sD_{[X]}^\lambda \Qc$ and ${D_{[X]}^\lambda}_0\Mo$.

\begin{lemma}
	Suppose that $\lambda \in \bZ\setminus\sA$ or that the greatest common divisor $\gcd_{i} (d_i) \neq 1$. Then $\Ke\Gamma_\lambda \neq 0$.
\end{lemma}

\begin{proof}
	If $k\in\bZ$,
	then $\sO_{[X]}(k)=\pi^\lambda_\bD(\bA[k])$
	is a non-zero $\bD^k$-saturated (since it is $\bA$-saturated \cite{AZ}) object of
	$\sD_{[X]}^k \Qc$
	because $1\in \bA_0=\bA[k]_{-k}$
	and
	$$\vE\cdot 1 =0 = (-k+k)1.$$
	The global sections
	$$\Gamma_k (\sO_{[X]}(k)) = \bA[-k]_0 = \bA_k$$ are non-zero if and only if
	$k\in \sA$. Thus, if $\lambda\in \bZ\setminus\sA$, then
	$\sO_{[X]}(\lambda)$
	is a non-zero object of
	$\Ke\Gamma_\lambda$.
	
	Now let us assume that the greatest common divisor
	$d$  of $d_0,...,d_n$ is greater than $1$. 
	It easily follows that $$\bD_1=\bD_2=\ldots  =\bD_{d-1}=0.$$
	Let $M$ be the $\bK$-vector space with a basis
	of all formal monomials $\vx_0^{a_0}\ldots \vx_n^{a_n}$,
	$a_i\in\bK$. 
	It is a $\bD$-module under the following operations, defined 
	on the monomials by
	\begin{align*}
	\vx_i \cdot \vx_0^{a_0}\ldots \vx_n^{a_n}
	&= \vx_0^{a_0}\ldots
	\vx_i^{1+a_i}\vx_{i+1}^{a_{i+1}}\ldots \vx_n^{a_n}, \\
	\partial_i \cdot \vx_0^{a_0}\ldots \vx_n^{a_n}
	&= a_i \vx_0^{a_0}\ldots
	\vx_i^{-1+a_i}\vx_{i+1}^{a_{i+1}}\ldots \vx_n^{a_n}.
	\end{align*}
	Given $\lambda \in \bK$, we consider the $\bD$-submodule
	$N=\bD\vx_0^{(\lambda -1)/d_0}$. Since
	$$
	\vE \cdot \vx_0^{(\lambda -1)/d_0} = 
	d_0\vx_0\partial_0 \cdot \vx_0^{(\lambda -1)/d_0} = 
	(\lambda -1)\vx_0^{(\lambda -1)/d_0},
	$$
	the module $N$ is $\lambda$-Euler and
	$\vx_0^{(\lambda -1)/d_0} \in N^{\lambda -1} = N_{-1}$ in the 
	canonical $\lambda$-Euler
	grading. Put $\mathcal{N}=\pi^\lambda_\bD(N)$. By definition, $N$ is torsion-free. Denote by $\tau^\lambda_\bD$ the restriction of $\tau_\bA$ to $\bD\Gr^\lambda$. The long exact sequence \cite{AZ}
	$$
	0  \rightarrow \tau^\lambda_\bD(N) \rightarrow N \rightarrow \omega^\lambda_\bD\pi^\lambda_\bD(N) \rightarrow R^{1}\tau^\lambda_\bD(N) \rightarrow 0
	$$
	reduces to the short exact sequence
	$$
	0  \rightarrow N \rightarrow \omega^\lambda_\bD\pi^\lambda_\bD(N) \rightarrow R^{1}\tau^\lambda_\bD(N) \rightarrow 0.
	$$
	
	But $R^{1}\tau^\lambda_\bD(N)$ is a torsion $\bD$-module, hence it is a direct sum of copies of $\Delta$. The $\vE$-weights of $N$ are congruent to $-1$ modulo $d$ and the $\vE$-weights of the module $\Delta$ are congruent to $0$ modulo $d$. It follows that the short exact sequence splits and 
	$$
	\omega^\lambda_\bD\pi^\lambda_\bD(N) \cong N \oplus R^{1}\tau^\lambda_\bD(N).
	$$
	Since $\omega^\lambda_\bD\pi^\lambda_\bD(N)$ is torsion free, $\omega^\lambda_\bD\pi^\lambda_\bD(N) \cong N$ and $R^{1}\tau^\lambda_\bD(N)=0$. This means that $N$ is $\bD^\lambda$-saturated and 
	$$
	\Gamma_\lambda(\mathcal{N})=N_0=0.
	$$
	Hence, 
	$\mathcal{N}$ is 
	a non-zero object in $\Ke \Gamma_\lambda$.
	\end{proof}

In all the other cases the kernel is trivial.

\begin{lemma}
	\label{DMod-ker}
	Let us assume that the greatest common divisor
	$\gcd_{i} (d_i)$ is equal to $1$. 
	If 
	$\lambda \in (\bK \setminus \bZ) \cup \sA$, 
	then $\Ke \Gamma_\lambda$ is a zero category.
\end{lemma}
\begin{proof}
	Let 
	$m$ be the least common multiple of
	$d_{0}, \ldots, d_{n}$. Suppose that $\mathcal{M}$ is a non-zero object in $\sD_{[X]}^\lambda-\Qcoh$. Then $M:=\omega^\lambda_\bD(\mathcal{M})$ is a non-zero $\lambda$-Euler torsion-free $\bD$-module.
	We need to show that $M_0\neq 0$.
	Let us suppose that the contrary is true, i.e., $M_0= 0$.
	We proceed to arrive at a contradiction 
	via 
	a sequence of claims.
	\begin{claim}
		$M_{-mt}=0$ for any $t\in\bZ_{>0}$.
	\end{claim}
	\begin{subproof}
		If $a\in M_{-mt}$, then $\vx_{i}^{mt/d_{i}}\cdot a=0$ for 
		all
		$i=0,\ldots, n$
		since it is an element of $M_{0}$. Hence,  
		$a$ generates a torsion $\bD$-submodule of $M$ but $M$ is torsion-free. Hence $a=0$.
		\end{subproof}
	\begin{claim}
		$M_{-mt+kd_{i}}=0$ for all $i$
		and
		$0\leqslant k\leqslant\frac{mt}{d_{i}}$.
		In particular, $M_{-kd_{i}}=0$
		for all
		$k\geqslant0$.
	\end{claim}
	\begin{subproof}
		We proceed by induction. The case $k=0$ is Claim~1. 
		Assume that this is true for $k$, and let us prove it for $k+1$. 
		If $-mt+\left(k+1\right)d_{i}=0$, 
		then we are done. 
		Otherwise, let us pick a non-zero element 
		$a\in M_{-mt+\left(k+1\right)d_{i}}$. It follows that 
		$$\partial_{i}\cdot a\in M_{-mt+kd_{i}}$$
		which is zero by induction. 
		Moreover, $\vx_{i}^{-\left(k+1\right)+mt/d_i}\cdot a \in M_{0}$
		which is zero again. 
		Since 
		$$
		\left[\partial_{i},\vx_{i}^{-\left(k+1\right)+mt/d_i}\right]=
		\left(\frac{mt}{d_{i}}-\left(k+1\right)\right)
		\vx_{i}^{-\left(k+2\right)+mt/d_i},
		$$
		we conclude that
		$\vx_{i}^{-\left(k+2\right)+mt/d_i}\cdot a=0$.
		We can repeat this argument to conclude that
		$\vx_{i}^{-\left(k+l\right)+mt/d_i}\cdot a=0$
		for all positive $l$ with 
		$\frac{mt}{d_{i}}-\left(k+l\right)\geq 0$.
		In particular, $a = \vx_i^0\cdot a =0$.
		\end{subproof}
	\begin{claim}
		If $c_{0},...,c_{k}$
		are positive integers and
		$g$ is their greatest common divisor, 
		then
		there exist integers
		$r_{0}\leqslant0$, and
		$r_{1},\ldots, r_{k}\geqslant0$ 
		such that
		$r_{0}c_{0}+\ldots +r_{k}c_{k}=g$.
	\end{claim}
	\begin{subproof}
		Let $l$ be the least common multiple
		of $c_{0},...,c_{k}$.
		By the Euclidean
		algorithm 
		there exist integers
		$s_{0},\ldots, s_{k}$
		such that
		$$
		s_{0}c_{0}+\ldots +s_{k}c_{k}=1.
		$$
		Now we can add $-\frac{l}{c_0}c_0+\frac{l}{c_i}c_i=0$
		for various $i$ to this relations to get integers 
		$r_{0},...,r_{k}$ 
		such that
		$$
		r_{0}c_{0}+\ldots +r_{k}c_{k}=1
		$$
		and 
		$r_{1}, \ldots, r_{k}\geqslant0$.
		Inevitably,  $r_{0}\leqslant0$.
		\end{subproof}
	\begin{claim}
		For all integer $b_{0},\ldots, b_{l}\geqslant0$,
		$M_{-\left(b_{0}d_{0}+\ldots +b_{l}d_{l}\right)}=0$.
	\end{claim}
	\begin{subproof}
		We proceed by induction on $l$. The base case $l=0$ is Claim 2.
		Assume this is true for $l-1$. In particular, it is true if $b_i=0$
		for some $i$.
		
		Let $g_{l}=\gcd\left(d_{0},\ldots ,d_{l}\right)$ and fix a positive integer $k$.
		Consider a non-zero element $a\in M_{-kg_l}$.
		There exist 
		positive integers $c_0,c_1,\ldots, c_l$
		such that 
		$$\partial_{0}^{c_{0}}\cdot a=\partial_{1}^{c_{1}}\cdot a=\ldots=
		\partial_{l}^{c_{l}}\cdot a=0.$$
		Indeed, by Claim 3, there exist 
		$r_{i}\leqslant0$
		and $r_{0}, \ldots, r_{i-1},r_{i+1}, \ldots r_{l}\geqslant0$ 
		such that 
		$$
		r_{0}d_{0}+\ldots +r_{l}d_{l}=g_{l}
		$$
		Now if $c_{i}=-kr_{i}\geqslant0$, then 
		$$
		\partial_{i}^{c_{i}}\cdot a \in
		M_{-c_{i}d_{i}-kg_l}=
		M_{-k(r_{0}d_{0}+\ldots +r_{i-1}d_{i-1}+r_{i+1}d_{i+1}+\ldots +r_{l}d_{l})}
		=0,
		$$
		by induction.
		Let us consider the Weyl algebra 
		$$
		\widetilde{\bD} = 
		\bK \langle \vx_{0},\ldots, \vx_{l}, \partial_0, \ldots,\partial_l \rangle
		$$
		and its polynomial subalgebra
		$\widetilde{\bA} = \bK \left[\partial_0, \ldots , \partial_l\right]$.
		The $\widetilde{\bA}$-module $\widetilde{\bD} a$ is supported at zero, hence,
		it must be a direct sum of copies of 
		$\widetilde{\Delta} = \widetilde{\bD} \delta
		(\partial_0, \ldots, \partial_l) \cong
		\bK \left[\vx_0, \ldots ,\vx_l\right]$.
		It follows that 
		$$\vx_{0}^{b_{0}}\ldots \vx_{l}^{b_{l}}\cdot a\neq0
		\ 
		\mbox{ for all }
		\
		b_{0},\ldots , b_{l}\geqslant0
		.$$
		We want to determine for which $k$, we can find
		$b_{0},\ldots , b_{l}\geqslant0$
		such that 
		$\vx_{0}^{b_{0}}\ldots \vx_{l}^{b_{l}}\cdot a \in M_{0}=0$.
		We get a contradiction and hence $M_{-kg_l}=0$ for such $k$. 
		The condition is that 
		$$b_{0}d_{0}+\ldots +b_{l}d_{l}=kg_l,$$ 
		i.e. $kg_l\in\bZ_{\geqslant0}d_{0}+\bZ_{\geqslant0}d_{1}+\ldots +\bZ_{\geqslant0}d_{l}$.
		\end{subproof}
	
	In particular, it is true for $l=n$, i.e., $M_{-k}=0$ for all $k\in\sA$.
	Now let us finish the proof of the theorem. By Schur's Theorem 
	there exists\footnote{
		The smallest such $K$ is called the Frobenius number. 
		It is a NP-hard problem to find such $K$.
		There is no known
		closed formula that gives $K$ as a function of $d_{0},...,d_{n}$
		for $n\geqslant2$.}
	$K\geqslant0$ such that 
	$k\in \sA$
	for all $k> K$, in particular,
	$M_{-k}=0$
	for all $k> K$.
	Thus, $M$ is supported at zero as a
	$\bK \left[ \partial_0, \ldots  \partial_n \right]$-module.
	By Kashiwara's Theorem $M$ is a direct sum of copies of $\bA=\bK \left[ \vx_0, \ldots  \vx_n \right]$.
	If $\lambda\in\bK\setminus\bZ$ then $\bA$ is not $\lambda$-Euler.
	Thus, $M=0$.
	Finally, 
	if $\lambda\in\bZ$ then $\bA$ is $\lambda$-Euler.
	Moreover, as a graded module 
	$M$ is a direct sum of copies of $\bA[\lambda]$.
	Observe that $\bA[\lambda]_0=\bA_\lambda\neq 0$ if and only if 
	$\lambda\in\sA$.
	Thus, if $\lambda\in\sA$, then $M=0$ as well.
	\end{proof}

Combining the last two claims, we obtain a characterisation of the kernel of the global sections functor.
\begin{theorem}
	The greatest common divisor
	$\gcd_{i} (d_i)$ is equal to $1$ and 
	$\lambda \in (\bK \setminus \bZ) \cup \sA$ if and only if 
	$\Ke \Gamma_\lambda$ is a zero category.
\end{theorem}

Together with Theorem~\ref{DMod-eq} this gives the following
corollaries.
\begin{cor}
Let us suppose that 
$\lambda\in (\bK\setminus\bZ) \cup\sA$
and
$\gcd\left(d_{0},...,d_{n}\right)=1$.
Then 
$\Gamma_\lambda : \sD_{[X]}^\lambda\Qc \rightarrow D_{[X]_0}^\lambda\Mo$
is an equivalence of categories.
\end{cor}

In particular, we obtain a necessary and sufficient condition for a weighted projective stack to be D-affine.
\begin{cor}
  The weighted projective stack
  $[X]=[\bP(V)]$ is D-affine if and only if
$\gcd_{i} (d_i)$ is equal to $1$. 
\end{cor}
\begin{proof}
	D-affinity deals with the case of $\lambda =0$. $\Gamma_0$ is exact, and its kernel is zero if and only if $\gcd_{i} (d_i)$ is equal to $1$.
\end{proof}
A similar functor for varieties 
$$
\Gamma^\prime_\lambda : \sD_{X}^\lambda\Qc \rightarrow D_{[X]_0}^\lambda\Mo
$$
is studied by Van den Bergh \cite{MVB}.
It is instructive to compare it with
the push-forward functor
$$
\pi_\ast : \sD_{[X]}^\lambda\Qc \rightarrow \sD_{X}^\lambda\Qc .
$$
The functors $\Gamma^\prime_\lambda \pi_\ast$ and $\Gamma_\lambda$
are naturally equivalent, so we can conclude 
the final corollary.
\begin{cor}
Let us suppose that 
$\lambda\in\bK\setminus\bZ\cup\sA$
and
$\gcd_{i\neq j}\left(d_i\right)=1$ for every $j$ 
(the well-formedness condition).
Then 
the push-forward functor
$
\pi_\ast : \sD_{[X]}^\lambda\Qc \rightarrow \sD_{X}^\lambda\Qc 
$
is an equivalence of categories.
\end{cor}

It can be noticed as well that the condition of well-formedness is not required for a weighted projective stack to be D-affine. We only need the greatest common divisor of its weights to be equal to one to guarantee it. As varieties, this condition was added to prove D-affinity of weighted projective spaces.

\bibliography{ams_Dmod}{}
\bibliographystyle{amsplain}
\end{document}